\documentclass[11pt,a4paper]{article}

\usepackage{comment}
\usepackage{graphicx}
\usepackage{amsmath}
\usepackage{amsfonts}
\usepackage{amssymb}
\usepackage{enumerate}
\usepackage{lscape}
\usepackage{longtable}
\usepackage{rotating}
\usepackage{color}
\usepackage{url}
\usepackage{rotating}
\usepackage{algpseudocode}
\usepackage[ruled]{algorithm}
\usepackage{pgf,tikz}
\usepackage{subfig}
\usepackage[titletoc]{appendix}
\usepackage{wrapfig}
\usepackage{adjustbox}
\usepackage{mathrsfs}
\usepackage{pgfplots}
\usetikzlibrary{arrows}
\usepackage{color}

\usepackage{bbm}

\usepackage{booktabs}
\usepackage{multirow}
\usepackage{mathtools}
\usepackage{makecell}

\usepackage[sort,numbers]{natbib}
\usepackage{eqparbox}%

\numberwithin{equation}{section}%

\usepackage[symbol]{footmisc}%

\newtheorem{theorem}{Theorem}[section]

\newtheorem{lemma}[theorem]{Lemma}

\newtheorem{assumption}{Assumption}[section]

\newenvironment{proof}[1][Proof]{\noindent \textbf{#1.} }{\hfill$\Box$\par\medskip}

\newcommand{\beqn}[1]{\begin{align*}\label{#1}}
\newcommand{\eeqn}{\end{align*}}

\definecolor{darkgreen}{rgb}{0,0.6,0}
\definecolor{aau2}{rgb}{0.0, 0.5, 0.69}
\definecolor{aau3}{rgb}{0.0, 0.53, 0.74}
\definecolor{aau4}{rgb}{0.0, 0.48, 0.65}
\definecolor{aau5}{rgb}{0.0, 0.45, 0.73}
\definecolor{rsap}{RGB}{130, 36, 51}
\definecolor{gsap}{RGB}{112, 164, 137}
\definecolor{tud}{rgb}{0.43,0.73,0.11}
\definecolor{verde}{rgb}{0.33,0.53,0.11}

\definecolor{ttffqq}{rgb}{0.0, 0.48, 0.65} %{rgb}{0.43,0.73,0.11}
\definecolor{ffqqqq}{rgb}{0.0, 0.5, 0.69} %{rgb}{1,0,0}
\usetikzlibrary{arrows}

\tikzstyle{decision} = [diamond, draw, fill=blue!20,
text width=4.5em, text badly centered, node distance=3cm, inner sep=0pt]
\tikzstyle{block} = [rectangle, draw, fill=blue!20,
text centered, rounded corners, minimum height=4em]
\tikzstyle{line} = [draw, -latex']
\tikzstyle{cloud} = [draw, ellipse,fill=red!20, node distance=3cm,
minimum height=2em]
\tikzstyle{cloud2} = [draw, ellipse,fill=green!20, node distance=3cm,
minimum height=2em]

\newcommand{\diam}{\Delta}
\newcommand{\Lphif}{L}

\newcommand{\Fm}{F_m}
\pgfplotsset{compat=1.18}
\begin{document}
\allowdisplaybreaks
\title{Stochastic block coordinate and function alternation for multi-objective optimization and learning}

	\author{
		T. H. Tran\thanks{Department of Industrial and Systems Engineering, Lehigh University, Bethlehem, PA 18015-1582, USA ({\tt hht320@lehigh.edu}).}
		\and
		L. N. Vicente\thanks{Department of Industrial and Systems Engineering, Lehigh University, Bethlehem, PA 18015-1582, USA ({\tt lnv@lehigh.edu}).}
	}
	
	\maketitle
Multi-objective optimization is central to many engineering and machine learning applications, where multiple objectives must be optimized in balance. While multi-gradient based optimization methods combine these objectives in each step, such methods require computing gradients with respect to all variables at every iteration, resulting in high computational costs in large-scale settings. In this work, we propose a framework that simultaneously alternates the optimization of each objective and the (stochastic) gradient update with respect to each variable block. Our framework reduces per-iteration computational cost while enabling exploration of the Pareto front by allocating a prescribed number of gradient steps to each objective. We establish rigorous convergence guarantees across several stochastic smooth settings, including convex, non-convex, and Polyak-Łojasiewicz conditions, recovering classical convergence rates of single-objective methods. Numerical experiments demonstrate that our framework outperforms non-alternating methods on 
multi-target regression and produces a competitive Pareto front approximation, highlighting its computational efficiency and practical effectiveness.

%%%%%%%%%%%%%%%%%%%%%%%%%%%%%%%%%%%%%%%%%%%%%%%

%%%%%%%%%%%%%%%%%%%%%%%%%%%%%%%%%%%%%%%%%%%%%%%

\section{Introduction}\label{sec:int}

Multi-objective optimization (MOO) arises in numerous real-world applications from engineering to finance where several conflicting criteria must be evaluated simultaneously~\cite{KMiettinen_2012,MEhrgott_2005,marler2004survey,NGunantara_2018}. Because these objectives potentially compete, the goal is to find a set of Pareto optimal solutions where no single objective can be improved without degrading another. Identifying this set of efficient points enables decision-makers to navigate complex trade-offs effectively. Furthermore, in many modern large-scale applications, these competing objectives are inherently stochastic, and the underlying decision variables exhibit a natural block structure~\cite{wright2015coordinate,shi2016primer,Nesterov2012EfficiencyOC}.
In this paper, we consider the following problem: 
\begin{align*}
    \min_{x\in \mathbb{R}^n} F(x) = (f_1(x), \dots, f_q(x))  = (\mathbb{E}[g_1(x,\xi)], \ldots, \mathbb{E}[g_q(x, \xi)]),
\end{align*}
where %$X \subset \mathbb{R}^n$, and
the decision variable $x \in \mathbb{R}^n$ is partitioned into $s$ disjoint blocks $x = (x_1, \dots, x_s)$, and each block $x_i$ has dimension $n_i$ with $\sum_{i=1}^s n_i = n$. 
We consider the MOO setting where each individual objective function $f_k$ has a Lipschitz continuous
gradient.
Moreover, each objective function is stochastic and computing precise gradients is typically either impossible or requires prohibitive computational resources. To overcome this limitation, our framework assumes access to an unbiased stochastic estimator $\nabla g_k(x, \xi)$ for the true gradient of each $f_k$, where $\xi$ represents a random variable.

%A multi-objective problem is considered smooth when all of its component objective functions $f_i$ are continuously differentiable. Throughout the convergence analysis of our proposed algorithms, we rely on this smoothness assumption for all objectives. Moreover, because each objective function incorporates inherent parametric randomness, exact gradients or subgradients are typically either unavailable or computationally prohibitive to evaluate. To overcome this limitation, our framework utilizes an unbiased stochastic estimator $\nabla g_i(x, \xi)$ for the true gradient of each $f_i$, where $\xi$ represents a random variable. 
% In this stochastic MOO setting, optimality is fundamentally established through the principle of Pareto dominance. Specifically, a decision vector $x$ is said to weakly dominate another vector $y$ if $f_i(x) \leq f_i(y)$ for all $i \in \{1, \dots, q\}$, and $F(x) \neq F(y)$. Consequently, a point $x \in X$ is formally defined as a Pareto optimal (or strictly non-dominated) solution if no other feasible point exists that weakly dominates it. The collection of all such non-dominated solutions is denoted by $\mathcal{P}$, and its mapping in the objective space constitutes the Pareto front, defined as $F(\mathcal{P}) = \{F(x) : x \in \mathcal{P}\}$.

\subsection{Literature review and motivation}
Existing MOO literature generally divides methods into two main categories: \textit{a priori} and \textit{a posteriori} methods, depending on when the decision-maker's preferences are incorporated into the optimization process. 
In MOO, the \textit{a priori} approach integrates decision-making preferences beforehand to transform the multi-objective problem into a simpler single-objective one, which can then be solved using single-objective optimization algorithms. Common strategies in this category include the weighted-sum method~\cite{SGass_TSaaty_1955}, which assigns non-negative weights to combine the objectives into a single convex linear combination; the $\epsilon$-constrained method~\cite{YVHaimes_1971}, which minimizes one primary objective while imposing upper bounds on the others; and the utility function method~\cite{EKBrowning_MAZupan_2020, KMiettinen_2012}, which employs a scalar function to mathematically quantify the overall preference between different objective vectors.
In contrast, \textit{a posteriori} methods are designed to generate a comprehensive or representative set of Pareto optimal solutions from which the decision-maker later selects the best option. These methods typically update a list of candidate solutions using either metaheuristics (such as evolutionary algorithms~\cite{SBechikh_RDatta_AGupta_2016, CCCoello_2006}) or rigorous descent mechanisms that compute common descent directions for all objectives simultaneously~\cite{ALCustodio_etal_2011, JADesideri_2012, EHFukuda_LMGDrummond_2014}. The alternating optimization algorithms proposed in this paper belong, in a certain sense, to both categories. % Rather than seeking a common descent direction to decrease all objectives simultaneously, our approach optimizes them sequentially. As we will demonstrate, these alternating algorithms implicitly optimize a specific convex linear combination of the objective functions.

In smooth, deterministic MOO, the multi-gradient algorithm applies the steepest descent step by computing the minimal-norm convex combination of gradients from individual functions. Therefore, a natural approach for stochastic MOO is the stochastic multi-gradient algorithm~\cite{MQuentin_PFabrice_JADesideri_2018,SLiu_LNVicente_2019}, where the gradients in the minimal-norm convex combination are replaced by its stochastic estimators. It is shown that the stochastic multi-gradient is a biased estimator of the true multi-gradient of the weighted function, and one can impose a bound on the amount of bias estimation by dynamic sampling~\cite{SLiu_LNVicente_2019}. With that assumption, 
\cite{SLiu_LNVicente_2019} establishes sublinear convergence rates that are similar to the ones known for single-objective optimization:~$\mathcal{O}(1/T)$ in the strongly convex case and~$\mathcal{O}(1/T^{1/2})$ in the convex one, where $T$ is the number of iterations.
While~\cite{SLiu_LNVicente_2019} only considers convex and strongly convex settings, follow-up work
\cite{SZhou_etal_2022,HFernando_etal_2023,LChen_HFernando_YYing_TChen_2023,HFernando_etal_2024} investigates multi-objective stochastic gradient methods in non-convex case, using different strategies to address the bias in the stochastic multi-gradient.
Without variance reduction,~\cite{LChen_HFernando_YYing_TChen_2023} shows the convergence rate of $\mathcal{O}(1/T^{1/2})$ for their proposed stochastic algorithm with double sampling in the non-convex case. 
% by a correlation reduction technique on the weight vector~$\lambda$ and updating it using the weight from the previous iteration. Meanwhile,~\cite{HFernando_etal_2023} proposes keeping track of variables to approximate the true gradient of each individual function to reduce the bias. Unlike these references,~\cite{LChen_HFernando_YYing_TChen_2023} does not require bounded function value assumptions, and utilizes a double sampling technique to obtain unbiased estimates of the Hessian matrix appearing in a modified version of subproblem~\eqref{prob:SMG}. The most recent work~\cite{HFernando_etal_2024} employ a variance reduction technique and improves the convergence for non-convex settings to $\mathcal{O}(1/k^{2/3})$, which matches the best known rate for single-objective optimization.

In single-objective optimization, block coordinate descent refers to a family of algorithms that update one block of the decision variable at each iteration~\cite{wright2015coordinate,shi2016primer,ABeck_LTetruashvili_2013,JLiu_etal_2014}. This approach is attractive when the variable naturally decomposes into blocks of coordinates and when full-gradient steps are expensive. The update of block variables can be done in either a Gauss-Seidel sequential style or a Gauss-Jacobi parallel style~\cite{ABeck_LTetruashvili_2013,JLiu_etal_2014,cai2023cyclic}. %The Gauss-Seidel approach updates the parameters using the most recent values of all other blocks, while the Jacobi style updates simultaneously based on the values from the previous iteration, which is more suitable for distributed computing. 
The order of such block variables can be chosen using a cyclic fashion, a random permutation of all blocks, or a randomized selection, where the new block is chosen uniformly at random. 
The convergence of deterministic BCD has been analyzed extensively in the literature~\cite{luo1992convergence,blockGauss-Seidel,razaviyayn2013unified,Nesterov2012EfficiencyOC,ABeck_LTetruashvili_2013,wen2012block}, where~\cite{Nesterov2012EfficiencyOC, ABeck_LTetruashvili_2013} consider general smooth convex settings for cyclic BCD and its randomized version, and achieve a convergence rate of $\mathcal{O}(1/T^{1/2})$.
The work~\cite{xu2015blockstochasticgradientiteration} pioneered the generalization of BCD for stochastic, single-objective optimization with theoretical results for both convex and non-convex regimes. The update order of the variable blocks can be fixed or shuffled, and the algorithm achieves sublinear convergence rates for convex and strongly convex  cases (namely $\mathcal{O}(1/T^{1/2})$ and $\mathcal{O}(1/T)$, respectively), which match the rates of stochastic gradient descent. However,~\cite{xu2015blockstochasticgradientiteration} only shows asymptotic convergence in the non-convex setting.  Later work~\cite{cuevas2017cyclic,jiang2025stochastic,yu2019zeroth,driggs2021stochastic,yang2019inexact,cai2023cyclic} expands this research direction, where~\cite{cai2023cyclic} provides non-asymptotic convergence guarantees for cyclic BCD methods in non-convex settings, however, utilizing a variance reduction technique.

\subsection{Contributions of this paper}

In this paper, we develop a stochastic block coordinate and function alternation algorithm (named Block-SMOO) for stochastic MOO. Block-SMOO has a two-loop structure where the outer loop iterates over the blocks of decision variables, while the inner loop applies a specific number of stochastic gradient descent steps for each objective function.
Our paper analyzes the theoretical convergence of Block-SMOO in various standard smooth settings, which recovers the classical convergence rates of standard single-objective stochastic gradient algorithms. Notably, as each update of our method only involves one objective function and one block of variables, it incurs a lower computational cost per iteration compared to standard full-gradient approaches like the weighted-sum methods~\cite{SGass_TSaaty_1955,SLiu_LNVicente_2019}. 

Block-SMOO differs from existing stochastic MOO literature, where the most closely related work is the stochastic alternating bi-objective algorithm~\cite{SLiu_LNVicente_2023}. While~\cite{SLiu_LNVicente_2023} only alternates function minimization, our method considers the simultaneous alternation of both functions and variables. Furthermore,~\cite{SLiu_LNVicente_2023} focuses on the specific case of two objective functions in convex and strongly convex settings, whereas our theory accommodates a general number of objective functions and provides theoretical guarantees across general convex, non-convex, and Polyak-Łojasiewicz (PL) conditions. Finally, our proof technique fundamentally differs from~\cite{SLiu_LNVicente_2023} as we establish convergence by utilizing a recursive descent bound for the objective function value, while the approach in~\cite{SLiu_LNVicente_2023} relies on an application of the Intermediate Value Theorem to aggregate the separate optimization steps.

In addition, our paper addresses a critical gap in the block coordinate descent literature for stochastic MOO, as no prior work has analyzed the simultaneous alternation of functions and variable blocks. The closely related BCD work~\cite{PCheng_etal_2024} employs scalarization to combine objectives into a single weighted function prior to optimization, rather than treating them separately, but it lacks a theoretical analysis for its proposed method. Our work resolves this by maintaining separated objectives and providing solid theory.

Meanwhile, our numerical experiments demonstrate the practical advantages of our method. When applied to compute an equally weighted Pareto optimal solution, Block-SMOO outperforms standard baseline approaches on multi-target regression tasks. We also validate the Block-SMOO's ability to approximate the full Pareto front, showing that it produces a trade-off surface structurally similar to the traditional weighted-sum approach but with a higher accuracy.
The remainder of this paper is organized as follows. In Section~\ref{sec:alg}, we provide a detailed description of the proposed Block-SMOO algorithm. We then present, in Section~\ref{sec:ana}, the theoretical convergence rate analysis of our method for the smooth, non-convex case.  In addition, we derive convergence results for the smooth convex case and the case where the functions satisfy the Polyak-Łojasiewicz (PL) condition. In Section~\ref{sec:exp}, we report the numerical performance of our algorithm on multi-target regression tasks, demonstrating its practical performance. The paper is concluded with some final remarks in Section~\ref{sec:con}.

\section{Stochastic block coordinate and function alternation algorithm}\label{sec:alg}

Described in Algorithm~\ref{alg_stoc}, our Block-SMOO operates by applying at each iteration a certain number of stochastic (partial) gradient steps to each individual function, where the order of those steps (meaning the selection of the individual functions) follows a certain user-specified sequence. 

Firstly, the algorithm is initialized with a budget of total gradient steps to be applied at each iteration, denoted by $p$. 
To specify which of the $q$ individual functions is selected for a gradient step in the order $0, \dots, p-1$, we introduce an index mapping $\pi : \{0, \dots, p-1\} \to \{1, \dots, q\}$. 
For any gradient step $j \in \{ 0, \dots, p-1\}$, the value $\pi(j)$ identifies which individual function is selected for such a step.
Next, we introduce a frequency vector $m \in \mathbb{N}^q$, where each component~$m_k$ specifies the number of gradient steps that is allocated to the corresponding objective function~$f_k$ within a single cycle from 0 to $p-1$. Consequently, the predetermined budget $p$ is also the sum of these individual frequencies,  $p = \sum_{k=1}^q m_k$.
One can then see that the index mapping $\pi$ %is %$\pi : \{0, \dots, p-1\} \to \{1, \dots, q\}$ is later 
%such that the cardinality of its preimage for objective $f_k$ satisfies $|\pi^{-1}(k)| = m_k$. 
is such that %$|\pi^{-1}(k)| = m_k$, 
the cardinality of its preimage for the objective $f_k$ is $m_k$, meaning $|\pi^{-1}(k)| = m_k$.
This cycle of gradient steps naturally introduces the following weighted-sum function, for which the weights are identified through the vector $m$:
\begin{align*}
    \Fm (\cdot) := \sum_{k=1}^q \frac{m_k}{p} f_k(\cdot) = \frac{1}{p}\sum_{j=0}^{p-1}f_{\pi(j)} (\cdot),
\end{align*}
where the second equality holds for every index mapping $\pi$ due to its definition. The Block-SMOO algorithm aims to minimize this weighted-sum function. 
To illustrate, consider a bi-objective setting ($q=2$) with a frequency vector $m = (5, 15)$, corresponding to a budget of $p = 20$. This configuration allocates 25\% of the steps to the first objective and 75\% to the second. As an example, one valid index mapping $\pi$ is the contiguous sequence 
\begin{align*}
    (\underbrace{1, \dots, 1}_{5 \text{ times}}, \underbrace{2, \dots, 2}_{15 \text{ times}}),
\end{align*}
and any arbitrary permutation of this sequence satisfies our algorithmic construction. As a result, the Block-SMOO will minimize the weighted-sum function $F_m = 0.25 f_1 + 0.75 f_2$. 

By varying the convex linear weights in $F_m$, one is guaranteed to find all Pareto solutions when the functions are convex~\cite{KMiettinen_2012,MEhrgott_2005}.
Therefore, by changing the frequencies $m_k$ for each function $f_k$ in the beginning of Block-SMOO, one can capture the trade-off between all objective functions. 
% We give 
% a detailed description of our stochastic multi-objective optimization method in Algorithm~\ref{alg_stoc}. 

\begin{algorithm}[ht]
\caption{Block-SMOO Algorithm} 
\label{alg_stoc} 
\begin{algorithmic}[1]
\item \textbf{Input:} Initial point $x^{0,0,0}$ and a step size sequence $\{\alpha_t\}$. \\
The budget $p$ and the frequency vector $m \in \mathbb{N}^q$, $p = \sum_{k=1}^q m_k$.
%Choose an index mapping $\pi: \{0, \dots, p-1\} \to \{1,\dots, q\}$.
\item {\bf for} $t = 0, \ldots, T-1$ {\bf do}
\item \quad Choose a permutation $\sigma$ of the block index set $\{1,\dots, s\}$.
\item \quad {\bf for} $i = 0, \ldots, s-1$ {\bf do}
\item \quad \quad Choose an index mapping $\pi: \{0, \dots, p-1\} \to \{1,\dots, q\}$ satisfying that $|\pi^{-1}(k)| = m_k$.
\item \quad \quad {\bf for} $j = 0, \ldots, p-1$ {\bf do}
\item \quad \quad \quad Generate a stochastic gradient $ \nabla_{\sigma(i+1)} g_{\pi(j)} (x^{t,i,j}, \xi^{t,i,j})$ to estimate $\nabla_{\sigma(i+1)} f_{\pi(j)} (x^{t,i,j})$. 
\item \quad \quad \quad Update $x_{\sigma(i+1)}^{t,i,j+1} = x_{\sigma(i+1)}^{t,i,j} - \alpha_t \nabla_{\sigma(i+1)} g_{\pi(j)} (x^{t,i,j}, \xi^{t,i,j})$.
\item \quad \quad \quad Set $x_{k}^{t,i,j+1} = x_{k}^{t,i,j}$ for all $k \neq \sigma(i+1)$.
\item \quad \quad Set $x^{t,i+1,0} = x^{t,i,p}$. 
\item \quad Set $x^{t+1,0,0} = x^{t,s,p}$.
\par\vspace*{0.1cm}
\end{algorithmic}
\end{algorithm}

After initialization, for each outer iteration $t$ of Algorithm~\ref{alg_stoc}, a permutation $\sigma$ of the block index set $\{1, \dots, s\}$ is chosen to randomize the order of the variable blocks.
Then, for each selected block $\sigma(i+1)$, it constructs an index mapping $\pi$
as described above. The algorithm then proceeds to update the chosen block $\sigma(i+1)$ using a stochastic gradient step of the individual function determined by $\pi(j)$. This process is repeated for every mapped sequence of objective ($j = 0, \dots, p-1$) and every block ($i = 0, \dots, s-1$), after which the algorithm advances to the next outer iteration. We note that our algorithm and convergence analyses allow for different permutations of $\sigma$ and $\pi$ across iterations, however, we will drop the indices $t, i$ throughout our paper for simplicity.

It is worth noting that Block-SMOO generalizes several well-known optimization algorithms. First, in the single-objective case, when $q=1$, our algorithm naturally reduces to the stochastic Block Coordinate Descent (BCD) method \cite{xu2015blockstochasticgradientiteration}. Second, when $s=1$, i.e., the decision variable is treated as a single block, the algorithm simplifies to an alternating function method for MOO (matching what was proposed in~\cite{SLiu_LNVicente_2023} for two objectives). Also, when~$s=1$, one can see our approach as a way to replicate the final effect of the weighted-sum method~\cite{SGass_TSaaty_1955}. Finally, in the most basic case where both $s=1$ and $q=1$, Block-SMOO recovers the classical stochastic gradient descent algorithm for single-objective optimization.

\section{Convergence analysis}\label{sec:ana}
Our convergence analysis starts with the non-convex and convex cases where we show the convergence rates $\mathcal{O}(1/T^{1/2})$ for both cases. We follow up with the setting of Polyak-Łojasiewicz condition, where the  convergence rate improves to $\mathcal{O}(1/T)$. 
\subsection{The non-convex case}
We first assume that all of the objective functions are Lipschitz smooth, and the weighted-sum function $F_m$ has a lower bound, which are standard assumptions in the optimization literature~\cite{bertsekas1997nonlinear}. 
\begin{assumption}[$\Lphif$-smoothness and Lower Bound]\label{assump:L}
For every $k \in \{ 1, \dots, q\}$, the function 
$f_k(x)$ is $\Lphif$-smooth, i.e., there exists a constant $\Lphif>0$ such that, $\forall x, x' \in \mathbb{R}^n$,
\begin{align*}
\|\nabla f_k(x) - \nabla f_k(x')\| \le \Lphif \|x - x'\|.
\end{align*}
%Trang: check proof, maybe this is L smooth wrt each block
% \item For every $j \in \{ 1, \dots, q\}$, the function $f_j(x)$ is $\Lf$-continuous, i.e., there exists a constant $\Lf>0$ such that, $\forall x, x' \in \mathbb{R}^n$,
% \begin{align*}
% | f_j(x) -  f_j(x')| \le \Lf \|x - x'\|.
% \end{align*}
% \end{enumerate}
In addition, we %denote $\Fm (\cdot) := \frac{1}{p}\sum_{j=0}^{p-1}f_{\pi(j)} (\cdot)$ and 
assume that $\Fm$ is bounded from below, i.e., $F^* := \inf_{x\in \mathbb{R}^n} \Fm(x) \in \mathbb{R}$. 

\end{assumption}
For the stochastic gradients $\nabla g_k(x,\xi), k \in \{ 1, \dots, q\}$, generated with random variable~$\xi$, we use~$\mathbb{E}_{\xi}[\cdot]$ to denote the conditional expectation taken with respect to $\xi$. In addition to Assumption~\ref{assump:L}, we impose the following two classical assumptions in stochastic approximation.  
\begin{assumption}[Unbiasedness and Bounded Gradient]\label{assump:conds}
For every $k \in \{ 1, \dots, q\}$, every $x \in \mathbb{R}^n$,
and every realization of $\xi$, the stochastic gradient $\nabla g_k(x,\xi)$ satisfies the following conditions
\begin{itemize}
\item[(a)] %Unbiasedness: 
$\mathbb{E}_{\xi} \left[\nabla g_k(x,\xi)\right] = \nabla f_k(x)$.
\item[(b)] %Bounded gradient: 
$\mathbb{E}_{\xi}[\|\nabla g_k(x,\xi)\|^2] \leq \sigma^2$, where $\sigma^2$ is a constant. 
\end{itemize} 
\end{assumption}

The above assumptions are commonly used for analyzing gradient-type methods~\cite{LBottou_FECurtis_JNocedal_2018,shalev2007pegasos,recht2011hogwild,nemirovski2009robust}, which guarantees that the stochastic gradient is unbiased and bounded. Note that the Block-SMOO algorithm does not need access to a stochastic estimator of the full gradient, and as a result, Assumption~\ref{assump:conds} only needs to hold for each partial stochastic gradient. We assume the full gradient for ease of notations. 
% \begin{definition}\label{def:minimizer}
% \textbf{Notations.}
% We denote the following: 
% \begin{align*}
%     \Fm (\cdot) &:= \frac{1}{p}\sum_{j=0}^{p-1}f_{\pi(j)} (\cdot), \text{ and } 
%     x^* := \arg\min_x \Fm (x).
% \end{align*}
% In addition, for every $j \in \{ 1, \dots, q\}$,
% \begin{align*}
%     x_j^* &:= \arg\min_{x} f_j(x) , \text{ and } 
%     \Fm^{**} := \frac{1}{p}\sum_{j=0}^{p-1} f_{\pi(j)}(x_{\pi(j)}^*).
% \end{align*}

% \end{definition}

Our main result below shows a bound for the expected squared norm gradient of the weighted-sum function $F_m$, which proves the sublinear convergence rate of $\mathcal{O}(1/T^{1/2})$. 
\begin{theorem}
\label{thm:non-convex}
Let Assumptions \ref{assump:L} and \ref{assump:conds} hold.
Let the step size be $\alpha_t = {1}/{\sqrt{T}}$ and $\bar{x}$ be the output of Algorithm~\ref{alg_stoc}. We  sample $\bar{x}$ uniformly at random from the set $\{x^{0,0,0}, \dots ,x^{T-1,0,0}\}$  with probability $\mathbb{P} (\bar{x} = x^{t,0,0}) = \frac{1}{T}$ for every $t$. %$t = 0, \dots, T-1$. %Let \dots be the output of Algorithm~\ref{alg_stoc}, 
We have
\begin{align*}
%\frac{1}{T}\sum_{t=0}^{T-1}  \mathbb{E} \left[\left\| \nabla \Fm(x^{t,0,0})  \right\|^2 \right] 
\mathbb{E} \left[\left\| \nabla \Fm(\bar{x})  \right\|^2 \right] 
&\leq  \frac{2}{\sqrt{T}} \left(\mathbb{E} [\Fm(x^{0,0,0})- F^*]  + \Lphif \sigma^2 p^2\left[s + \frac{\Lphif s^3}{3}\right]\right).
\end{align*}

\end{theorem}
% Under the assumption that each individual objective function $f_k$ is smooth and that the stochastic gradients are bounded and unbiased estimators, we can establish that $\Fm$ is also $\Lphif$-smooth. 
We present the proof of Theorem~\ref{thm:non-convex} in Section~\ref{sec:non-convex} of the Appendix, which we briefly summarize below. Our non-convex analysis 
%The proof of Theorem~\ref{thm:non-convex} 
begins by bounding the weighted-sum function $\Fm(x^{t,i,j+1})$ by $\Fm(x^{t,i,j})$, i.e., 
\begin{align*}
\Fm(x^{t,i,j+1}) &\leq \Fm(x^{t,i,j}) + \nabla \Fm(x^{t,i,j})^\top (x^{t,i,j+1} - x^{t,i,j}) + \frac{\Lphif}{2} \|x^{t,i,j+1} - x^{t,i,j}\|^2,
\end{align*}
where we use the fact that $\Fm$ is also $\Lphif$-smooth. This observation, along with applications of Assumption~\ref{assump:conds} leads to the following recursive bound for the expected function value 
\begin{align*}%\label{eq:bound_non-convex}
\mathbb{E} [\Fm(x^{t,s,p})]
&\leq \mathbb{E} [\Fm(x^{t,0,0})] - \alpha_t  \mathbb{E} \left[\sum_{i=0}^{s-1}\left\| \nabla_{\sigma(i+1)} \Fm(x^{t,i,0})  \right\|^2 \right]   + ps\alpha_t^2 \frac{\Lphif}{2} \sigma^2 + \alpha_t \mathbb{E} [A^t]. 
\end{align*}
As the gradients are evaluated at intermediate points $x^{t,i,j}$ rather than $x^{t,i,0}$, in this bound, we introduce an error term $A^t$ that considers these deviations,
\begin{align*}
&\left|\sum_{i=0}^{s-1}\nabla_{\sigma(i+1)} \Fm(x^{t,i,0})^\top \nabla_{\sigma(i+1)} \Fm(x^{t,i,0}) -\sum_{i=0}^{s-1}\sum_{j=0}^{p-1}\nabla_{\sigma(i+1)} \Fm(x^{t,i,j})^\top  \nabla_{\sigma(i+1)} f_{\pi(j)} (x^{t,i,j})\right|,
\end{align*}
and we show that $\mathbb{E}[A^t]  \leq  \alpha_t \Lphif \sigma^2 s p(p-1)$ using 
$\Lphif$-smoothness and bounded gradient assumptions.

Substituting this error bound for $A^t$ back into the recursive bound and employing a relation between $\nabla \Fm(x^{t,i,0}) $ and $\nabla \Fm(x^{t,0,0})$, we arrive at the following descent bound for the algorithm progress per outer iteration:
% \textcolor{blue}{
% [I will add a summary of proof until it gets to the descent bound (1) below, as we use it for the strongly convex/ PL result]}
\begin{align}\label{eq:main_bound2}
\mathbb{E} [\Fm(x^{t+1,0,0})]
&\leq \mathbb{E} [\Fm(x^{t,0,0})]  -\frac{1}{2} \alpha_t \mathbb{E} \left[\left\| \nabla \Fm(x^{t,0,0})  \right\|^2 \right]   + \alpha_t^2 \Lphif \sigma^2 p^2\left[s + \frac{\Lphif \alpha_t  s^3}{3}\right],
\end{align}
where $x^{t+1,0,0} = x^{t,s,p}$ from Algorithm~\ref{alg_stoc}. 
Applying the fixed step size $\alpha_t = \frac{1}{\sqrt{T}}$ and averaging for all $t$, we obtain
\begin{align*}
\frac{1}{T}\sum_{t=0}^{T-1}  \mathbb{E} \left[\left\| \nabla \Fm(x^{t,0,0})  \right\|^2 \right] 
&\leq  \frac{2}{\sqrt{T}} \left(\mathbb{E} [\Fm(x^{0,0,0})- F^*]  + \Lphif \sigma^2 p^2\left[s + \frac{\Lphif s^3}{3}\right]\right),
\end{align*}
where the final bound of Theorem~\ref{thm:non-convex} follows from the definition of $\bar{x}$. We note that this practice of choosing random output point $\bar{x}$ is standard in the literature of stochastic optimization \cite{ghadimi2013stochastic}.

Our convergence rate for Block-SMOO recovers the classical convergence rate $\mathcal{O}(1/T^{1/2})$ of stochastic gradient descent for single-objective optimization \cite{ghadimi2013stochastic}.
We note that for MOO, the  work~\cite{SLiu_LNVicente_2019} does not consider the non-convex setting.  Meanwhile, in single-objective optimization,~\cite{xu2015blockstochasticgradientiteration} only considers asymptotic theoretical non-convex results for their stochastic BCD algorithm.  
% [Comments on the rate of MOO alternating functions and SGD]
The stochastic setting in~\cite{xu2015blockstochasticgradientiteration} is somewhat similar to our Assumption~\ref{assump:conds}, however, their non-convex  theory assumes a vanishing variance of the stochastic gradient. 

\subsection{The convex case}

Before presenting the convex theoretical analysis, we impose the additional assumptions as follows. 
\begin{assumption}[Convexity and Bounded Iterations]\label{assump:convex}
We assume that $f_k(x)$ is convex for every $k \in \{ 1, \dots, q\}$ and the solution 
$x^* := \arg\min_x \Fm (x)$ exists. 
%For every $j \in \{ 1, \dots, q\}$, 
% $f_j(x)$ is convex. %, i.e., 
% $\forall x, x' \in \mathbb{R}^n$,
% \begin{align*}
% f_j(x) - f_j(x') \ge \langle \nabla f_j(x'), x - x' \rangle.
% \end{align*}
In addition, we assume that 
$\left\|x^{t,0,0}- x^* \right\| \leq \diam$, for every outer iteration $t$ of Algorithm~\ref{alg_stoc}. 
\end{assumption}

% \begin{remark} 
The iterates generated by Algorithm~\ref{alg_stoc} are assumed to remain bounded. This is a standard assumption widely adopted in the literature for convex optimization and block coordinate descent methods~\cite{Nesterov2012EfficiencyOC,ABeck_LTetruashvili_2013,xu2015blockstochasticgradientiteration,SLiu_LNVicente_2019}. We note that instead of relying on this assumption, the algorithm can be modified to include a projection step onto a bounded convex set after each update. 
The first paragraph of our convex analysis can be easily adapted using the non-expansiveness property of orthogonal projections.

We are now ready to present our convex theorem, which bounds the expected optimality gap for the weighted-sum function $F_m$. 
\begin{theorem}
\label{thm:convex}
Let Assumptions \ref{assump:L}, \ref{assump:conds}, and \ref{assump:convex} hold. Let the step size be $\alpha_t = {1}/{\sqrt{T}}$ and 
$\bar{x} = \frac{1}{T}\sum_{t=0}^{T-1} x^{t,0,0}$ be the output of Algorithm~\ref{alg_stoc}.
We have
\begin{align*}
 \mathbb{E}[\Fm(\bar{x}) - F^*] 
    &\leq \frac{\|x^{0,0,0}- x^*\|^2 + sp \sigma[\sigma+  (\diam \Lphif +  \sigma) sp])}{2p\sqrt{T}}  . 
\end{align*}

\end{theorem}
The proof of Theorem~\ref{thm:convex} is postponed to  Section~\ref{sec:convex} of the Appendix. 
% [If we have space after the experiment I will add a short summary of proof]
Our convex analysis first establishes the following recursive bound
\begin{align*}
\mathbb{E}[\|x^{t+1,0,0} - x^*\|^2]
    &\leq \mathbb{E}[\|x^{t,0,0}- x^*\|^2]- 2\alpha_t p\mathbb{E}[(x^{t,0,0}- x^*)^\top(\nabla \Fm(x^{t,0,0}) )]\\
    &+ \alpha^2_t sp \sigma[\sigma+  (\diam \Lphif +  \sigma) sp].
\end{align*}
Then we apply the fact that all objective functions are convex, therefore $\Fm$ is convex, which yields
\begin{align*}
\mathbb{E}[\|x^{t+1,0,0} - x^*\|^2]
    &\leq \mathbb{E}[\|x^{t,0,0}- x^*\|^2] - 2\alpha_t p\mathbb{E}[\Fm(x^{t,0,0}) - F^*] + \alpha^2_t sp \sigma[\sigma+  (\diam \Lphif +  \sigma) sp],
\end{align*}
from where we obtain the final bound of Theorem~\ref{thm:convex} by applying a fixed step size and averaging the derived recursive bound.

One can see from the proof that the result of Theorem~\ref{thm:convex} only requires convexity of $\Fm$. When the goal is to approximate the whole Pareto front, one can apply Algorithm~\ref{alg_stoc} for a collection of values of the frequency vector $m$ so that $m/p$ discretizes sufficiently well the simplex of dimension~$q$. In doing that, $\Fm$ has to be convex for all $m$, which accounts to say that all the individual functions are convex (as in Assumption~\ref{assump:convex}). 

Theorem~\ref{thm:convex} recovers the standard convergence rate $\mathcal{O}(1/T^{1/2})$ of BCD single-objective optimization \cite{ghadimi2013stochastic, xu2015blockstochasticgradientiteration} and of stochastic alternating function minimization for 
bi-objective optimization~\cite{SLiu_LNVicente_2019}. 

% [Will comment a bit on how our assumptions differ from these work. maybe ours is better in some sense]

\subsection{The Polyak-Łojasiewicz case}
In this section, we consider the Polyak-Łojasiewicz (PL) inequality, a generalization of strong-convexity \citep{Polyak1964,polyak_condition}, which we impose for the weighted-sum function $F_m$. 
% \begin{assumption}[Strong Convexity]\label{assump:strongly-convex}
% For every $j \in \{ 1, \dots, q\}$, 
% $f_j(x)$ is $\mu$-strongly convex, i.e., there exists a constant $\mu > 0$ such that 
% $\forall x, x' \in \mathbb{R}^n$,
% \begin{align*}
% f_j(x) - f_j(x') \ge \langle \nabla f_j(x'), x - x' \rangle + \frac{\mu}{2}\|x - x'\|^2.
% \end{align*}
% \end{assumption}
% Trang: maybe use PL condition as it is non-convex and more preferable by the ML community... 
%  Luis: comment after the assumption \& remark about projection in convex case.
\begin{assumption}[Polyak-Łojasiewicz (PL) condition]\label{assump:PL}
The function $\Fm$ satisfies the  $\mu$-PL inequality, i.e., there exists a constant $\mu > 0$ such that 
$\forall x \in \mathbb{R}^n$,
\begin{align*}
    \left\| \nabla \Fm(x)  \right\|^2 \geq 2\mu [\Fm(x) - F^*], 
\end{align*}
where $F^*$ is defined as in Assumption \ref{assump:L}.
\end{assumption}

It is well known that a function satisfying the PL condition is not necessarily convex \citep{polyak_condition}. Under this assumption, one can show that single-objective stochastic gradient descent achieves the same theoretical rate as the sublinear rate $\mathcal{O}(1/T)$ in the strongly convex setting \citep{pmlr-v54-de17a,gower2021sgd}. Similarly, we show the same convergence rate for the optimality gap of Block-SMOO as follows. 
\begin{theorem}
\label{thm:stronglyconvex}
Let Assumptions \ref{assump:L}, \ref{assump:conds}, and \ref{assump:PL} hold.
Let the step size be $\alpha_t = \frac{2}{\mu (t+1)}$ and $\bar{x} = x^{T,0,0}$ be the output of Algorithm~\ref{alg_stoc}.
We have
\begin{align*}
\mathbb{E} [\Fm(\bar{x})- F^*] %= \mathbb{E} [\Fm(x^{T,0,0})- F^*]
&\leq  \frac{4\Lphif \sigma^2 p^2 (3\mu s + 2\Lphif  s^3)}{3T\mu^3}.
\end{align*}

\end{theorem}

The proof of Theorem~\ref{thm:stronglyconvex} follows from the bound~\eqref{eq:main_bound2}, derived for the non-convex case.
\vspace{\baselineskip}

\begin{proof}
From Assumption \ref{assump:PL}, one has
\begin{align*}
    \left\| \nabla \Fm(x^{t,0,0})  \right\|^2 \geq 2\mu [\Fm(x^{t,0,0}) - F^*].
\end{align*}
Combining this bound with the inequality~\eqref{eq:main_bound2}, for every $t$, we obtain
\begin{align*}
\mathbb{E} [\Fm(x^{t+1,0,0})]
&\leq \mathbb{E} [\Fm(x^{t,0,0})]  -\frac{1}{2} \alpha_t \mathbb{E} \left[2\mu [\Fm(x^{t,0,0}) - F^*]\right]   + \alpha_t^2 \Lphif \sigma^2 p^2\left[s + \frac{\Lphif \alpha_t  s^3}{3}\right],
\end{align*}
which is equivalent to 
\begin{align*}
\mathbb{E} [\Fm(x^{t+1,0,0})- F^*]
&\leq (1- \alpha_t \mu)\mathbb{E} [\Fm(x^{t,0,0}) - F^*]  + \alpha_t^2 \Lphif \sigma^2 p^2\left[s + \frac{\Lphif \alpha_t  s^3}{3}\right].
\end{align*}
Choosing $\alpha_t = \frac{2}{\mu (t+1)}$ results in
\begin{align*}
\mathbb{E} [\Fm(x^{t+1,0,0})- F^*]
&\leq \frac{t-1}{t+1}\mathbb{E} [\Fm(x^{t,0,0}) - F^*]  + \frac{4}{\mu^2 (t+1)^2} \Lphif \sigma^2 p^2\left[s + \frac{2\Lphif  s^3}{3(t+1)\mu }\right], 
\end{align*}
which leads to 
\begin{align*}
t(t+1)\mathbb{E} [\Fm(x^{t+1,0,0})- F^*]
&\leq t(t-1)\mathbb{E} [\Fm(x^{t,0,0}) - F^*]  + \frac{4t}{\mu^2 (t+1)} \Lphif \sigma^2 p^2\left[s + \frac{2\Lphif  s^3}{3\mu }\right].
\end{align*}

Applying the inequality recursively for all $t$ yields
\begin{align*}
(T-1)T\mathbb{E} [\Fm(x^{T,0,0})- F^*]
&\leq \sum_{t=0}^{T-1} \frac{4t}{\mu^2 (t+1)} \Lphif \sigma^2 p^2\left[s + \frac{2\Lphif  s^3}{3\mu }\right],\\
&\leq (T-1) \frac{4}{\mu^2} \Lphif \sigma^2 p^2\left[s + \frac{2\Lphif  s^3}{3\mu }\right].
\end{align*}
Dividing both sides by $T(T-1)$, we obtain the desired bound. 
\end{proof}
The Polyak-Łojasiewicz convergence rate analysis of Block-SMOO matches the strongly convex convergence rate $\mathcal{O}(1/T^{1/2})$ of BCD single-objective optimization \cite{ghadimi2013stochastic, xu2015blockstochasticgradientiteration} and the stochastic alternating function minimization for 
bi-objective optimization~\cite{SLiu_LNVicente_2019}. However, the Polyak-Łojasiewicz condition is weaker than assuming strong convexity, and in this case we do not need to assume bounded iterates as in~\cite{SLiu_LNVicente_2019}.

\section{Experiments}\label{sec:exp}
 
% We consider the multi-variate reduced-rank regression problem: 
% \begin{align*}
%     \min_{A, B} \| Y - ABX \|^2, 
% \end{align*}
% where $Y \in \mathbb{R}^q$, $X \in \mathbb{R}^d$, $A \in \mathbb{R}^q \times \mathbb{R}^r$, $B \in \mathbb{R}^r \times \mathbb{R}^d$ and $r$ is a specified low rank dimension.
In this section, we illustrate the empirical performance of Block-SMOO, comparing it to other non-alternating methods. 
\subsection{Reduced-rank regression problem as an MOO}
We consider the following multivariate reduced-rank regression problem \cite{izenman1975reduced,reinsel1998multivariate}:
\begin{align}\label{prob:rrr}
    \min_{U,V} \left\| Y - XUV \right\|^2,
\end{align}
where $X \in \mathbb{R}^{N \times d}$ is the feature matrix, $Y \in \mathbb{R}^{N \times q}$ is the response matrix, $U \in \mathbb{R}^{d \times r}$ and $V \in \mathbb{R}^{r \times q}$ are the low-rank factor matrices with prescribed rank $r$.

Multivariate regression is widely used to model multiple correlated outcomes simultaneously, e.g., in econometrics, biology, and computer science~\cite{johnson2002applied,reinsel1998multivariate}. Meanwhile, reduced rank regression considers the case when the variables share a low-dimensional structure, which has applications in finance, neuroimaging, and high-dimensional statistics \cite{izenman1975reduced,reinsel1998multivariate,vounou2010discovering}. As regressions for each of the $q$ responses may convey different meanings, this problem can be recast as multi-objective optimization by treating each of the response as a separate objective, yielding
\begin{align*}
    \min_{U,V} F(U,V) = (f_1(U,V), \dots, f_q(U,V)),  \text{ where } f_k(U,V) =
     \left\| Y_k - XUV_k \right\|^2, k \in\{ 1, \dots, q\},
\end{align*}
where $Y_k \in \mathbb{R}^N$ and $V_k \in \mathbb{R}^r$ denote the $k$-th column of $Y$ and $V$, respectively. If one minimizes the equal-weighted-sum function $F_{1/q} = 1/ q \sum_{k=1}^q f_k$ for this MOO, then the problem reduces to the reduced-rank regression problem~\eqref{prob:rrr}. As our problem setting is stochastic, we compute  gradient estimators from a minibatch of the dataset, which we describe below. 

\subsection{Datasets}
We consider two datasets in this experiment.  
Firstly, 
we generate a \textit{synthetic} dataset with $N_{\mathrm{train}} = 2^{14}$ training samples and $N_{\mathrm{test}} = 2^{10}$ test samples, input dimension $d = 400$, and number of responses $q = 5$. The ground-truth weight matrix $W^{\star} = U^{\star} V^{\star}$ has rank $r = 3$, where $U^{\star} \in \mathbb{R}^{d \times r}, V^{\star} \in \mathbb{R}^{r \times q}$ are drawn i.i.d.\ from $\mathcal{N}(0, I)$. Observations are generated as
\begin{align}
    Y = X U^{\star} V^{\star} + \varepsilon, \qquad \varepsilon \sim \mathcal{N}(0, \sigma^2 I),
\end{align}
with noise level $\sigma = 0.05$, and the data $X$ are drawn i.i.d.\ from $\mathcal{N}(0, I)$.

In addition, 
we perform experiments on the \textit{Beijing Multi-Site Air Quality} dataset~\cite{beijing_air_quality}, publicly available from the UCI Machine Learning Repository~\cite{uci_ml_repository}. The dataset contains hourly measurements from 12 monitoring stations across Beijing from 2013 to 2017. The response matrix $Y$ consists of the concentrations of six air pollutants: PM$_{2.5}$, PM$_{10}$, SO$_2$, NO$_2$, CO, and~O$_3$. The feature matrix $X$ comprises meteorological and temporal covariates including wind direction, wind speed, temperature, dew point, pressure, precipitation, and station and time indicators, yielding $d = 35$ features after one-hot encoding of categorical variables and removal of rows with missing values.
We split the data chronologically; the first $70\%$ of rows forms the training set and the remaining $30\%$ forms the test set ($N_{\mathrm{train}} \approx 265{,}000$ and $N_{\mathrm{test}} \approx 115{,}000$). All features and responses are standardized to zero mean and unit variance using training-set statistics.

% Finally, we evaluate on the \textit{SARCOS inverse dynamics} dataset~\cite{sarcosVijayakumar2000}, available from the repository~\cite{gpmlData}. This dataset corresponds to inverse dynamics regression for a seven degrees-of-freedom anthropomorphic robot arm. Each sample contains 21 input variables (7 joint positions, 7 joint velocities, and 7 joint accelerations) and 7 output variables (joint torques), i.e., $X \in \mathbb{R}^{n \times 21}$ and $Y \in \mathbb{R}^{n \times 7}$.
% The dataset provides predefined train and test matrices with 44,484 training samples and 4,449 test samples. We also standardize both features and targets using training-set statistics.

\subsection{Methods and experiment setting}

We evaluate our proposed Block-SMOO algorithm against three baseline methods. To ensure a fair comparison, all algorithms optimize the same equal-weighted-sum function~$F_{1/q} = \frac{1}{q} \sum_{k=1}^q f_k$. The first baseline, denoted as the \text{Weighted-Sum} method~\cite{SGass_TSaaty_1955}, employs standard mini-batch stochastic gradient descent on $F_{1/q}$, updating all functions and coordinates simultaneously at every step. The second baseline is the \text{Function-Alternate} method (Algorithm~\ref{alg_stoc} with $s=1$), which implements stochastic alternating function minimization, a generalization of the bi-objective approach in~\cite{SLiu_LNVicente_2019}. The third baseline, termed the \text{Block-Alternate} method (Algorithm~\ref{alg_stoc} with $q=1$), applies the stochastic BCD algorithm~\cite{xu2015blockstochasticgradientiteration} directly to the weighted-sum function. Finally, our proposed \text{Block-SMOO} method alternates the minimization of both functions and variable blocks (Algorithm~\ref{alg_stoc}). For Block-SMOO, we report results using a preference vector where $m_k = 2$ for all $k$, as this configuration empirically proved most efficient among the tested values $m_k \in \{1, 2, 4, 8\}$.

Furthermore, decision variables can be partitioned into blocks in various ways. In our experiments for both \text{Block-SMOO} and \text{Block-Alternate}, we evaluate two partitioning strategies: a two-block partition $\{U, V\}$ and a four-block partition $\{U, V^{1}, V^{2}, V^{3}\}$, where $V^{i}$ denotes the $i$-th row of $V$. We report results from the best-performing partition for each algorithm. Additionally, for all alternating variants, the index ordering of~$\sigma$ and~$\pi$ in Algorithm~\ref{alg_stoc} are randomly reshuffled at the beginning of each corresponding cycle.

All algorithms use a batch size of %$B = 128$ for the synthetic dataset and 
$B = 512$ for all %the real 
datasets. The initial matrices $U$ and~$V$ are sampled independently from $\mathcal{N}(0, 0.01I)$.
Each configuration is evaluated over a grid of step sizes, i.e., $\{0.001,\, 0.002,\, 0.005, 0.01,\, 0.02,\, 0.05\}$,
% \begin{align*}
%     \eta \in \{0.001,\, 0.002,\, 0.005, 0.01,\, 0.02,\, 0.05\}.
% \end{align*},
and the best step size for each algorithm is selected as the one achieving the lowest final test loss across the grid.
We note that our method requires less computational cost per iteration, and thus for fair comparison, we run 
all algorithms for a fixed budget of time $T_{\max} = 2\,\mathrm{sec}$ for the \textit{synthetic} dataset and $T_{\max} = 4\,\mathrm{sec}$ for the \textit{Beijing Multi-Site Air Quality} dataset, with test losses recorded every $\Delta t = 0.2\,\mathrm{sec}$ and $\Delta t = 0.4\,\mathrm{sec}$, respectively for the two datasets. Results are averaged over 10 independent random seeds. 
% All the experiments were conducted on a MacBook Pro with a 2.3 GHz Quad-Core Intel Core i7 processor and 16 GB of memory.
\subsection{Comparison results}
\begin{figure}[H]
  \centering  \includegraphics[width=0.495 \textwidth]{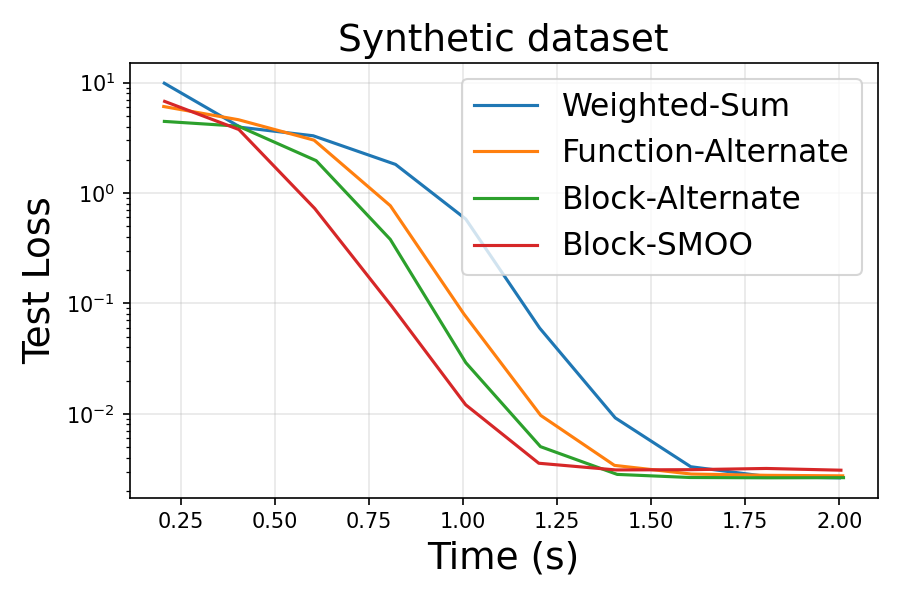}
    \includegraphics[width=0.495 \textwidth]{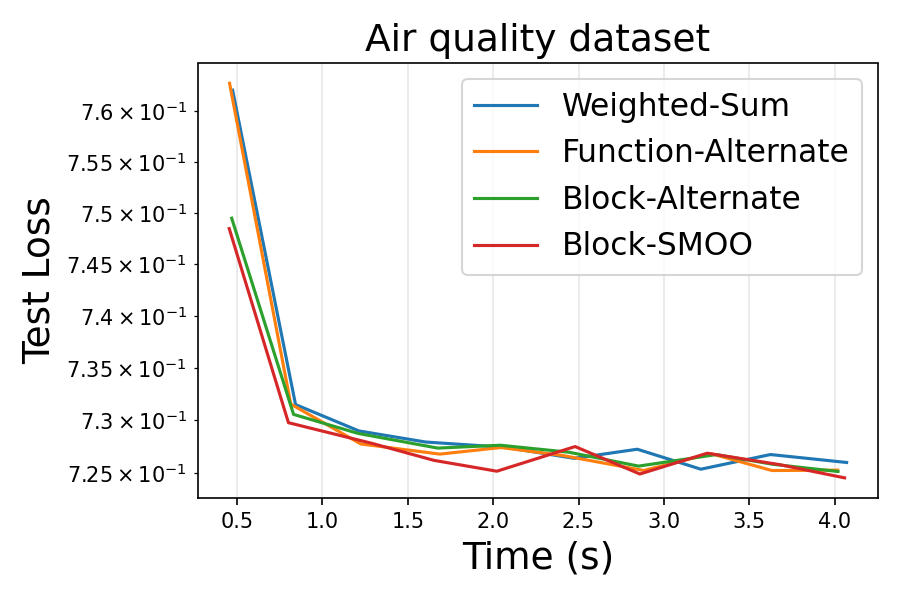}
  \caption{Test loss computed on the weighted-sum function $F_{1/q}$ versus computer time on the \textit{synthetic} dataset (left) and the \textit{Beijing Multi-Site Air Quality} dataset (right).}\label{fig2}
\end{figure}
% \begin{figure}[htbp]
% \centering
% % \includegraphics[width=0.7\textwidth]{synthetic_results.png} % Path to your synthetic dataset figure (e.g., 3.png)
% \caption{Test loss versus wall-clock time on the synthetic dataset. Block-SMOO achieves the fastest initial convergence, reaching the optimal loss significantly quicker than the Weighted-Sum baseline, which requires roughly twice the runtime to fully converge.}
% \label{fig:synthetic}
% \end{figure}

We report the test loss versus time for the datasets in Figure \ref{fig2}.
On the \textit{synthetic dataset}, our proposed Block-SMOO algorithm exhibits the most rapid initial convergence. It drops to a low loss significantly quicker than the standard Weighted-Sum method. While the Function-Alternate and Block-Alternate methods offer moderate improvements over the Weighted-Sum baseline, they are consistently outpaced during the initial descent by the dual-alternating approach of Block-SMOO.

On the real-world \textit{Beijing Air Quality} dataset, Block-SMOO continues to demonstrate robust empirical performance. It consistently achieves the lowest test loss during the early stages of optimization and maintains a comparable performance over the baseline methods throughout the entire training duration. The standard Weighted-Sum method generally exhibits the slowest progress and higher overall test loss early in the run. These empirical findings support our theoretical claims and highlight the practical effectiveness gained by simultaneously alternating the update steps of objective functions and variable blocks.

\subsection{Pareto front approximation}

We are now interested in the ability of Block-SMOO to determine the whole Pareto front, by applying it to a collection of values of the vector $m$. For the \textit{Beijing Multi-Site Air Quality} dataset above, we considered an MOO setting with only three objective functions, corresponding to the air pollutants: PM$_{2.5}$, PM$_{10}$, and SO$_2$. We compare the approximated Pareto front generated by Block-SMOO to the one from the Weighted-Sum method. 

We ran Block-SMOO $231$ times with $p = m_1 + m_2 + m_3 = 20$. Weighted-Sum was also run $231$ times on the objective $F_m = 1/20(m_1 f_1 + m_2 f_2 + m_3 f_3)$, where in both cases $m_k \in \{0,\dots, 20\}$ for all $k$.
For both approaches, we use a fixed step size $0.02$ at each iteration. %The starting point for each run is randomly generated within the feasible region. 
Similar to the prior experiment, we sample the initial matrices $U$ and $V$ independently from $\mathcal{N}(0, 0.01 I)$.
We run each algorithm for $20$ data passes. As this Pareto front experiment is computationally expensive, we truncated the train set and the test set so that they contains $N_{\mathrm{train}} = 2^{14}$ training samples and $N_{\mathrm{test}} = 2^{10}$ test samples, respectively.

\begin{figure}[H]
  \centering  \includegraphics[width= \textwidth]{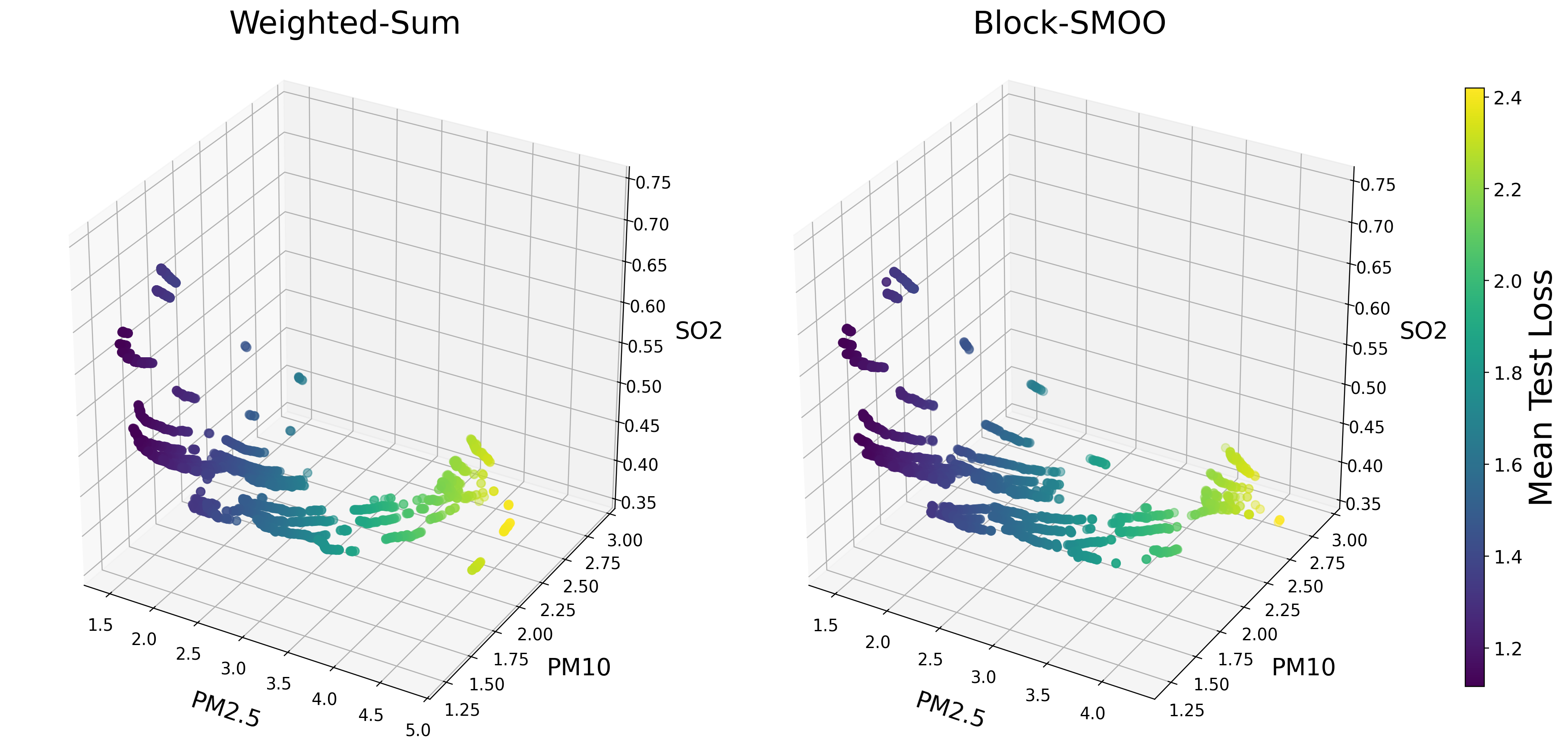}
  \caption{Approximation of Pareto fronts for test losses computed by the Weighted-Sum (left) and Block-SMOO (right) algorithms for the \textit{Beijing Multi-Site Air Quality} dataset.}\label{Pareto_fronts}
\end{figure}
Figure~\ref{Pareto_fronts}
illustrates the approximated Pareto fronts obtained from both the Weighted-Sum and the Block-SMOO algorithms for the described problem. We computed the average test loss over three objective functions for the gradient color of each plot. Both plots demonstrate a similar structure of non-dominated solutions, suggesting that Block-SMOO effectively captures the trade-off in the Pareto front, while maintaining competitive performance across multiple objectives compared to the standard Weighted-Sum method.
\begin{table}[H]
\caption{Metrics for Pareto fronts computed by the Weighted-Sum and Block-SMOO methods.}
\label{tab:metrics_comparison}
\begin{center}
\begin{tabular}{l|ccc}
\toprule
Metric &  Purity & Spread $\Gamma$ & Spread $\Delta$ \\
\midrule
Weighted-Sum   & 0.5105 & 0.2878 & 1.1549\\
Block-SMOO   & 0.6947 & 0.4955 & 1.1785\\
\bottomrule
\end{tabular} 
\end{center}
\end{table}
Moreover, to assess the quality of the generated Pareto fronts, we employed two standard metrics: Purity \cite{SBandyopadhya_SKPal_BAruna_2004} and Spread \cite{KDeb_etal_2002}. The Purity metric evaluates the accuracy of the approximation by calculating the percentage of ``true" non-dominated solutions among all the non-dominated points generated by each algorithm. A higher purity ratio corresponds
to a more accurate Pareto front. 
On the other hand, the Spread metric is designed to measure the extent of the point spread in a computed Pareto front. We computed two Spread variations: the $\Gamma$ metric, which measures the maximum gap or ``hole size" between points, and the $\Delta$ metric, which indicates how well the points are distributed. Lower values for both $\Gamma$ and $\Delta$ indicate a more evenly distributed and well-spread Pareto front. Detailed mathematical formula for computations of these metrics can be found in \cite{ALCustodio_etal_2011,SLiu_LNVicente_2019}.
We present the results in Table \ref{tab:metrics_comparison}.
Block-SMOO has achieved a superior accuracy with a Purity of approximately 70\% (compared to 51\% for Weighted-Sum), which is actually visible from Figure~\ref{Pareto_fronts}. The two methods yielded fronts with a similar value of the~$\Delta$~Spread. Weighted-Sum has however achieved a better~$\Gamma$~spread.

%and a larger maximum gap $\Gamma$. This indicates a clear trade-off between the depth (Purity) and the breadth (Spread) of the resulting Pareto fronts. To resolve this conflict, one can tune hyperparameters such as the number of initial starting points or the density of points generated at each iteration to achieve the balance between solution accuracy and Pareto front coverage.

% Metrics for Test Set:
% purity: {'alg_0': 0.5105193951347797, 'alg_1': 0.6947073256221521}
% spread_gamma: {'alg_0': 0.28783326462341474, 'alg_1': 0.4954577990798663}
% spread_delta: {'alg_0': 1.1548615789049206, 'alg_1': 1.1784880657762298}
\section{Concluding remarks}\label{sec:con}

In this paper, we demonstrated that the proposed stochastic block coordinate and function minimization algorithm (Block-SMOO) achieves convergence rates matching those of standard single-objective block coordinate descent methods. Specifically, we established sublinear convergence rates of $\mathcal{O}(1/T^{1/2})$ for both general smooth convex and non-convex functions, and an better rate of $\mathcal{O}(1/T)$ when the weighted-sum function satisfies the Polyak-Łojasiewicz (PL) condition. These theoretical guarantees are derived using applications of recursive descent bounds on the objective function value and the squared distance to optimality, established for the starting iterates in each outer loop of the algorithm.

Our current convergence analysis focuses on a two-loop structure where the outer loop iterates over the variable block indices, and the inner loop applies gradient steps to individual objective functions. Although our established bounds do not automatically hold if this two-loop structure is inverted, we anticipate that the underlying proof techniques and principles developed in this work can be adapted to that reversed setting with some minor theoretical modifications.

Other research questions arises from this work. While the current algorithm executes updates in a sequential Gauss-Seidel fashion, exploring parallel updates in a Gauss-Jacobi style represents an exciting, yet challenging direction that is particularly suitable for distributed computing problems. Furthermore, because many practical applications are characterized by non-smooth optimization landscapes, extending our framework to accommodate these cases and providing meaningful convergence guarantees for non-smooth multi-objective optimization is an important and promising topic for future research.

%%%%%%%%%%%%%%%%%%%%%%%%%%%%%%%%%%%%%%

% \section*{Acknowledgments}
% This work is partially supported by the U.S. Air Force Office of Scientific Research~(AFOSR) award~FA9550-23-1-0217 and the U.S. Office of Naval Research~(ONR) award~N000142412656.

% \bibliography{ref}

\appendix

%%%%%%%%%%%%%%%%%%%%%%%%%%%%%%%%%%%%%%%%%%%%%%%%%%%%%%%%%%%%%%%%%%%%%%%%%%%%%%%%%%%%%%%%%
\section*{Appendix}\label{app:}
%%%%%%%%%%%%%%%%%%%%%%%%%%%%%%%%%%%%%%%%%%%%%%%%%%%%%%%%%%%%%%%%%%%%%%%%%%%%%%%%%%%%%%%%%

\section{Basic notations and lemma}
In this Appendix, we present the proof of the non-convex case in Section~\ref{sec:non-convex} and the proof of the convex case in Section~\ref{sec:convex}. Throughout the proofs, similar to~\cite{Nesterov2012EfficiencyOC}, we use the notation  $U_i \in \mathbb{R}^{n \times n_i} $, for $ i=1,\dots,s$, where 
\begin{align}\label{eq:def_Ui}
    (U_1, U_2, \dots, U_s) = I_n. 
\end{align}
Then in this notation $x_i = U_i^\top x$ and $x = \sum_{i=1}^s U_i x_i$ for every $i \in \{ 1, \dots, s\}$ and $x \in \mathbb{R}^n$.
%The vector of partial derivatives corresponding to the variables in the vector $x_i$ is $U_i^\top \nabla f(x)$.

Recall that in Algorithm~\ref{alg_stoc}, we denote by  $\nabla_{\sigma(i+1)} f_{\pi(j)} (x^{t,i,j})$ the (partial) gradient of $f_{\pi(j)}$ with respect to $x_{\sigma(i+1)}$, computed at $x^{t,i,j}$. By the updates of Algorithm \ref{alg_stoc} and the definition of~$U_i$ in~\eqref{eq:def_Ui}, we have 
\begin{align}\label{eq:alg_update}
    x^{t,i,j+1} = x^{t,i,j} - \alpha_t U_{\sigma(i+1)} \nabla_{\sigma(i+1)} g_{\pi(j)} (x^{t,i,j}, \xi^{t,i,j}).
\end{align}
In addition, we let $\mathcal{F}_{t,i,j}$ be the $\sigma$-algebra generated by $\{x^{0,0,0}, \xi^{0,0,0}, \xi^{0,0,1}, \dots, \xi^{t,i,j-1}\}$ i.e., $\mathcal{F}_{t,i,j}$ contains all the random information up to iteration $(t,i,j)$. Before presenting our main analysis, we need the following Lemma. 
\begin{lemma}\label{lem:bounds}
Let Assumption \ref{assump:conds} holds. 
We have 
\begin{itemize}
\item[(a)] For every $x\in \mathbb{R}^n$, $\|\nabla f_j(x)\| \leq \sigma$ and $\|\nabla \Fm(x)\| \leq \sigma$.
\item[(b)] For every $t, i,$ and $j$, $\mathbb{E} \left\|x^{t,i,j+1}-x^{t,i,j}\right\| \leq \alpha_t \sigma $ and $\mathbb{E} \left[\left\|x^{t,i,j+1}-x^{t,i,i}\right\|^2\right] \leq \alpha^2_t \sigma^2$.
\end{itemize} 
\end{lemma}
\begin{proof} 
\begin{itemize}
\item[(a)] We obtain the first bound by applying the unbiased property in Assumption \ref{assump:conds}(a), the Jensen's inequality, and the bounded gradient condition in Assumption \ref{assump:conds}(b),
\begin{align}\label{eq:bound_grad}
    \|\nabla f_j(x)\| =  \|\mathbb{E}_{\xi} \left[\nabla g_j(x,\xi)\right]\| \leq \sqrt{ \mathbb{E}_{\xi} \left[\|\nabla g_j(x,\xi)\|^2\right]} \leq \sigma.
\end{align}
The second bound follows from the definition of $\Fm$ and the triangle inequality, i.e.,
\begin{align*}
    \|\nabla \Fm(x)\| =  \left\|\frac{1}{p}\sum_{j=0}^{p-1}\nabla f_{\pi(j)} (x)\right\| \leq \frac{1}{p}\sum_{j=0}^{p-1}\|\nabla f_{\pi(j)} (x)\| \overset{\eqref{eq:bound_grad}}{\leq} \sigma.
\end{align*}
%By convexity of the norm operator, $\|\nabla \Fm(x)\|  \leq \sigma$.
\item[(b)] 
From equation~\eqref{eq:alg_update}, we have 
\begin{align*}
    \|x^{t,i,j+1} - x^{t,i,j}\| = \alpha_t \| U_{\sigma(i+1)} \nabla_{\sigma(i+1)} g_{\pi(j)} (x^{t,i,j}, \xi^{t,i,j})\|.
\end{align*}
Applying expectation with respect to $\mathcal{F}_{t,i,j}$, then using Jensen's inequality, the fact that $\| U_{\sigma(i+1)}\| \leq 1$, and Assumption \ref{assump:conds}(b), one obtains 
\begin{align*}
    \mathbb{E} [\|x^{t,i,j+1} - x^{t,i,j}\|  | \mathcal{F}_{t,i,j}]&= \alpha_t  \mathbb{E} [\| U_{\sigma(i+1)} \nabla_{\sigma(i+1)} g_{\pi(j)} (x^{t,i,j}, \xi^{t,i,j})\|| \mathcal{F}_{t,i,j}]\\
    &\leq \alpha_t  \sqrt{\mathbb{E} [\| U_{\sigma(i+1)} \nabla_{\sigma(i+1)} g_{\pi(j)} (x^{t,i,j}, \xi^{t,i,j})\|^2| \mathcal{F}_{t,i,j}]} \leq \alpha_t \sigma.
\end{align*}
Similar arguments yield
\begin{align*}
    \mathbb{E} [\|x^{t,i,j+1} - x^{t,i,j}\|^2  | \mathcal{F}_{t,i,j}]&= \alpha_t^2  \mathbb{E} [\| U_{\sigma(i+1)} \nabla_{\sigma(i+1)} g_{\pi(j)} (x^{t,i,j}, \xi^{t,i,j})\|^2| \mathcal{F}_{t,i,j}]\leq \alpha_t^2 \sigma^2.
\end{align*}
Taking total expectation, we obtain
\begin{align*}
    \mathbb{E} \left\|x^{t,i,j+1}-x^{t,i,j}\right\| \leq \alpha_t \sigma \text{ and }
    \mathbb{E} \left[\left\|x^{t,i,j+1}-x^{t,i,j}\right\|^2\right] \leq \alpha_t^2 \sigma^2.
\end{align*}
\end{itemize} 
\end{proof}
\section{Proof of the non-convex case}\label{sec:non-convex}

We are ready to present the proof of Theorem~\ref{thm:non-convex} below. 
\vspace{\baselineskip}

\begin{proof}
From the definition of $\Fm$ and Assumption \ref{assump:L}, $\Fm$ is also $\Lphif$-smooth. Based on the descent lemma~\cite[Proposition A.24]{bertsekas1997nonlinear} and Equation~\eqref{eq:alg_update}, we have the following bound for every $t,i,j$, 
\begin{align*}
\Fm(x^{t,i,j+1}) &\leq \Fm(x^{t,i,j}) + \nabla \Fm(x^{t,i,j})^\top (x^{t,i,j+1} - x^{t,i,j}) + \frac{\Lphif}{2} \|x^{t,i,j+1} - x^{t,i,j}\|^2\\
&\leq \Fm(x^{t,i,j}) - \alpha_t \nabla \Fm(x^{t,i,j})^\top \left(  U_{\sigma(i+1)} \nabla_{\sigma(i+1)} g_{\pi(j)} (x^{t,i,j}, \xi^{t,i,j}) \right)  \\
&+ \alpha_t^2 \frac{\Lphif}{2} \left\|   U_{\sigma(i+1)} \nabla_{\sigma(i+1)} g_{\pi(j)} (x^{t,i,j}, \xi^{t,i,j})\right\|^2\\
&\leq \Fm(x^{t,i,j}) - \alpha_t  \nabla_{\sigma(i+1)} \Fm(x^{t,i,j})^\top \left(  \nabla_{\sigma(i+1)} g_{\pi(j)} (x^{t,i,j}, \xi^{t,i,j}) \right)  \\
&+ \alpha_t^2 \frac{\Lphif}{2} \left\| \nabla_{\sigma(i+1)} g_{\pi(j)} (x^{t,i,j}, \xi^{t,i,j})\right\|^2,
\end{align*}
where the last equation used  $\|U_{\sigma(i+1)}\|^2 \leq 1$ and $g_i^\top = g^\top U_i$. 
Applying expectation with respect to $\mathcal{F}_{t,i,j}$ to the previous bound followed by Assumption \ref{assump:conds}, one obtains
\begin{align*}
\mathbb{E} [\Fm(x^{t,i,j+1}) | \mathcal{F}_{t,i,j}]
&\leq \Fm(x^{t,i,j}) - \alpha_t  \nabla_{\sigma(i+1)} \Fm(x^{t,i,j})^\top  \nabla_{\sigma(i+1)} f_{\pi(j)} (x^{t,i,j})  + \alpha_t^2 \frac{\Lphif}{2} \sigma^2.
\end{align*}

Taking total expectation on both sides of this bound and applying the inequality recursively for all $i,j$, we have
\begin{align}\label{eq:bound_non-convex}
\mathbb{E} [\Fm(x^{t,s,p})]
&\leq \mathbb{E} [\Fm(x^{t,0,0})] - \alpha_t \sum_{i=0}^{s-1}\sum_{j=0}^{p-1} \mathbb{E} \left [\nabla_{\sigma(i+1)} \Fm(x^{t,i,j})^\top  \nabla_{\sigma(i+1)} f_{\pi(j)} (x^{t,i,j})\right]  + ps\alpha_t^2 \frac{\Lphif}{2} \sigma^2\nonumber\\
&\leq \mathbb{E} [\Fm(x^{t,0,0})] - \alpha_t  \mathbb{E} \left[\sum_{i=0}^{s-1}\left\| \nabla_{\sigma(i+1)} \Fm(x^{t,i,0})  \right\|^2 \right]   + ps\alpha_t^2 \frac{\Lphif}{2} \sigma^2 + \alpha_t \mathbb{E} [A^t], 
\end{align}
where $A^t$ is the following term:
\begin{align*}
&\left|\sum_{i=0}^{s-1}\nabla_{\sigma(i+1)} \Fm(x^{t,i,0})^\top \nabla_{\sigma(i+1)} \Fm(x^{t,i,0}) -\sum_{i=0}^{s-1}\sum_{j=0}^{p-1}\nabla_{\sigma(i+1)} \Fm(x^{t,i,j})^\top  \nabla_{\sigma(i+1)} f_{\pi(j)} (x^{t,i,j})\right|\\
&=\left|\sum_{i=0}^{s-1}\sum_{j=0}^{p-1}\nabla_{\sigma(i+1)} \Fm(x^{t,i,0})^\top  \nabla_{\sigma(i+1)} f_{\pi(j)} (x^{t,i,0}) -\sum_{i=0}^{s-1}\sum_{j=0}^{p-1}\nabla_{\sigma(i+1)} \Fm(x^{t,i,j})^\top  \nabla_{\sigma(i+1)} f_{\pi(j)} (x^{t,i,j})\right|\\
&\leq 
\left|\sum_{i=0}^{s-1}\sum_{j=0}^{p-1}\nabla_{\sigma(i+1)} \Fm(x^{t,i,0})^\top  \nabla_{\sigma(i+1)} f_{\pi(j)} (x^{t,i,0}) -\sum_{i=0}^{s-1}\sum_{j=0}^{p-1}\nabla_{\sigma(i+1)} \Fm(x^{t,i,0})^\top  \nabla_{\sigma(i+1)} f_{\pi(j)} (x^{t,i,j})\right|\\
&+\left|\sum_{i=0}^{s-1}\sum_{j=0}^{p-1}\nabla_{\sigma(i+1)} \Fm(x^{t,i,0})^\top  \nabla_{\sigma(i+1)} f_{\pi(j)} (x^{t,i,j}) -\sum_{i=0}^{s-1}\sum_{j=0}^{p-1}\nabla_{\sigma(i+1)} \Fm(x^{t,i,j})^\top  \nabla_{\sigma(i+1)} f_{\pi(j)} (x^{t,i,j})\right|\\
&\leq 
\sum_{i=0}^{s-1}\sum_{j=0}^{p-1}\left\| \nabla_{\sigma(i+1)} \Fm(x^{t,i,0})\right\| \cdot \left\| \nabla_{\sigma(i+1)} f_{\pi(j)} (x^{t,i,0})  -\nabla_{\sigma(i+1)} f_{\pi(j)} (x^{t,i,j}) \right\|\\
&+\sum_{i=0}^{s-1}\sum_{j=0}^{p-1}\left\| \nabla_{\sigma(i+1)} \Fm(x^{t,i,0}) -\nabla_{\sigma(i+1)} \Fm(x^{t,i,j})\right\|\cdot \left\|  \nabla_{\sigma(i+1)} f_{\pi(j)} (x^{t,i,j}) \right\|\\
&\leq 
\sum_{i=0}^{s-1}\sum_{j=0}^{p-1}\sigma \Lphif \left\| x^{t,i,0} -x^{t,i,j}\right\|+\sum_{i=0}^{s-1}\sum_{j=0}^{p-1} \sigma \Lphif \left  \|x^{t,i,0}-x^{t,i,j}\right\| = 2\sigma \Lphif \sum_{i=0}^{s-1} \sum_{j=0}^{p-1}  \left  \|x^{t,i,0}-x^{t,i,j}\right\|,
\end{align*}
where we used the $\Lphif$-smoothness and bounded gradient of $\Fm$ and $f_{\pi(j)}$.
Taking expectation of both sides, we arrive at
\begin{align*}
    \mathbb{E}[A^t] &
    \leq 2\sigma \Lphif  \mathbb{E} \left[ \sum_{i=0}^{s-1} \sum_{j=0}^{p-1}  \left  \|x^{t,i,0}-x^{t,i,j}\right\|\right]= 2\sigma \Lphif   \sum_{i=0}^{s-1} \sum_{j=0}^{p-1}  \mathbb{E} \left[\left  \|x^{t,i,0}-x^{t,i,j}\right\|\right] \\
    &\leq 2\sigma \Lphif   \sum_{i=0}^{s-1} \sum_{j=0}^{p-1}  \mathbb{E} \left[\sum_{k=0}^{j-1} \left  \|x^{t,i,k+1}-x^{t,i,k}\right\|\right] \leq  2\sigma \Lphif  \sum_{i=0}^{s-1}  \sum_{j=0}^{p-1} j\alpha_t\sigma \leq  \alpha_t \Lphif \sigma^2 s p(p-1),
\end{align*}
where the last line follows from the triangle inequality and Lemma \ref{lem:bounds}(b).

Substituting $\mathbb{E}[A^t]$ into our main derivation~\eqref{eq:bound_non-convex} yields
\begin{align*}
\mathbb{E} [\Fm(x^{t,s,p})]
&\leq \mathbb{E} [\Fm(x^{t,0,0})] - \alpha_t \mathbb{E} \left[\sum_{i=0}^{s-1}\left\| \nabla_{\sigma(i+1)} \Fm(x^{t,i,0})  \right\|^2 \right]   + ps\alpha_t^2 \frac{\Lphif}{2} \sigma^2 + \alpha_t^2 \Lphif \sigma^2s p(p-1)\\
&\leq \mathbb{E} [\Fm(x^{t,0,0})]  -\frac{1}{2} \alpha_t \mathbb{E} \left[\sum_{i=0}^{s-1}\left\| \nabla_{\sigma(i+1)} \Fm(x^{t,0,0})  \right\|^2 \right]  \\
&+ \alpha_t \mathbb{E} \left[\sum_{i=0}^{s-1}\left\| \nabla_{\sigma(i+1)} \Fm(x^{t,i,0}) -\nabla_{\sigma(i+1)} \Fm(x^{t,0,0})  \right\|^2 \right]+ \alpha_t^2 \Lphif \sigma^2 s \left[p(p-1)+ \frac{p}{2}\right],
\end{align*}
where we used the inequality $-a^2 \leq -\frac{1}{2} b^2 +(a-b)^2$. Using the $\Lphif$-smooth property of $\Fm$ and recalling from Algorithm~\ref{alg_stoc} that~$x^{t+1,0,0} = x^{t,s,p}$, one obtains
\begin{align*}%\label{eq:main_bound}
\mathbb{E} [\Fm(x^{t+1,0,0})]
&\leq \mathbb{E} [\Fm(x^{t,0,0})]  -\frac{1}{2} \alpha_t \mathbb{E} \left[\left\| \nabla \Fm(x^{t,0,0})  \right\|^2 \right]  \nonumber \\
&+ \alpha_t \Lphif^2\mathbb{E} \left[\sum_{i=0}^{s-1}\left\|x^{t,i,0} - x^{t,0,0}\right\|^2 \right]+ \alpha_t^2 \Lphif \sigma^2 sp^2 \nonumber \\
&\leq \mathbb{E} [\Fm(x^{t,0,0})]  -\frac{1}{2} \alpha_t \mathbb{E} \left[\left\| \nabla \Fm(x^{t,0,0})  \right\|^2 \right]  \nonumber \\
&+ \alpha_t \Lphif^2 \sum_{i=0}^{s-1} \mathbb{E} \left[ip \sum_{k=0}^{i-1} \sum_{j=0}^{p-1} \left\|  x^{t,k,j+1} - x^{t,k,j}\right\|^2 \right]+ \alpha_t^2 \Lphif \sigma^2 sp^2 \nonumber \\
&\leq \mathbb{E} [\Fm(x^{t,0,0})]  -\frac{1}{2} \alpha_t \mathbb{E} \left[\left\| \nabla \Fm(x^{t,0,0})  \right\|^2 \right] +\alpha_t\Lphif^2 \sum_{i=0}^{s-1}  i^2p^2\alpha^2_t \sigma^2+ \alpha_t^2 \Lphif \sigma^2 p^2 \nonumber \\
&\leq \mathbb{E} [\Fm(x^{t,0,0})]  -\frac{1}{2} \alpha_t \mathbb{E} \left[\left\| \nabla \Fm(x^{t,0,0})  \right\|^2 \right]   + \alpha_t^2 \Lphif \sigma^2 p^2\left[s + \frac{\Lphif \alpha_t  s^3}{3}\right],
\end{align*}
where the second bound follows from the inequality $\|\sum_{i=1}^n a_i\|^2 \leq n\sum_{i=1}^n \| a_i\|^2 $ and the third bound follows from  Lemma \ref{lem:bounds}(b).
Rearranging the derived bound gives us
\begin{align*}
\mathbb{E} \left[\left\| \nabla \Fm(x^{t,0,0})  \right\|^2 \right] 
&\leq  \frac{2}{\alpha_t} \left(   \mathbb{E} [\Fm(x^{t,0,0})] - \mathbb{E} [\Fm(x^{t+1,0,0})] \right)  + 2\alpha_t \Lphif \sigma^2 p^2\left[s + \frac{\Lphif \alpha_t  s^3}{3}\right].
\end{align*}

Averaging for all $t$, one obtains
\begin{align*}
\frac{1}{T}\sum_{t=0}^{T-1} \mathbb{E} \left[\left\| \nabla \Fm(x^{t,0,0})  \right\|^2 \right] 
&\leq  \frac{1}{T}\sum_{t=0}^{T-1} \frac{2}{\alpha_t} \left(   \mathbb{E} [\Fm(x^{t,0,0})] - \mathbb{E} [\Fm(x^{t+1,0,0})] \right) \\
&+ \frac{1}{T}\sum_{t=0}^{T-1} 2\alpha_t \Lphif \sigma^2 p^2\left[s + \frac{\Lphif \alpha_t  s^3}{3}\right].
\end{align*}
Setting $\alpha_t = \frac{1}{\sqrt{T}}$ yields
\begin{align*}
\frac{1}{T}\sum_{t=0}^{T-1}  \mathbb{E} \left[\left\| \nabla \Fm(x^{t,0,0})  \right\|^2 \right] 
&\leq  \frac{2}{\sqrt{T}} \left(\mathbb{E} [\Fm(x^{0,0,0})- F^*]  + \Lphif \sigma^2 p^2\left[s + \frac{\Lphif s^3}{3}\right]\right),
\end{align*}
which concludes the proof.  
\end{proof}

\section{Proof of the convex case}\label{sec:convex}
We present the proof of Theorem~\ref{thm:convex} as follows.
\vspace{\baselineskip}

\begin{proof} For every iteration $t, i, j$, applying  Equation~\eqref{eq:alg_update} yields
\begin{align*}
    \|x^{t,i,j+1} - x^*\|^2 
    &= \|(x^{t,i,j}- x^*)  + (x^{t,i,j+1} - x^{t,i,j})\|^2 \\
    &= \|x^{t,i,j}- x^*\|^2 +2 (x^{t,i,j}- x^*)^\top(x^{t,i,j+1} - x^{t,i,j}) + \|x^{t,i,j+1} - x^{t,i,j}\|^2\\
    &= \|x^{t,i,j}- x^*\|^2 -2\alpha_t (x^{t,i,j}- x^*)^\top( U_{\sigma(i+1)} \nabla_{\sigma(i+1)} g_{\pi(j)} (x^{t,i,j}, \xi^{t,i,j})) \\
    &+ \alpha_t^2\|U_{\sigma(i+1)} \nabla_{\sigma(i+1)} g_{\pi(j)} (x^{t,i,j}, \xi^{t,i,j})\|^2.
\end{align*}
Applying expectation with respect to $\mathcal{F}_{t,i,j}$ on both sides of the previous bound  and using Assumption \ref{assump:conds}, we obtain
\begin{align*}
\mathbb{E}[\|x^{t,i,j+1} - x^*\|^2 | \mathcal{F}_{t,i,j}]
    &\leq \|x^{t,i,j}- x^*\|^2 -2\alpha_t (x^{t,i,j}- x^*)^\top( U_{\sigma(i+1)} \nabla_{\sigma(i+1)} f_{\pi(j)} (x^{t,i,j})) + \alpha_t^2 \sigma^2.
\end{align*}

Taking total expectation of the previous bound and applying the inequality recursively for all $i,j$, one has
\begin{align}\label{eq:bound_convex}
&\mathbb{E}[\|x^{t,s,p} - x^*\|^2]\nonumber \\
    &\leq \mathbb{E}[\|x^{t,0,0}- x^*\|^2] - 2\alpha_t\sum_{i=0}^{s-1}\sum_{j=0}^{p-1} \mathbb{E}[(x^{t,i,j}- x^*)^\top U_{\sigma(i+1)} \nabla_{\sigma(i+1)} f_{\pi(j)} (x^{t,i,j}) ] + sp\alpha_t^2 \sigma^2\nonumber \\
    &\leq \mathbb{E}[\|x^{t,0,0}- x^*\|^2] - 2\alpha_t p\mathbb{E}[(x^{t,0,0}- x^*)^\top(\nabla \Fm(x^{t,0,0}) )] + sp\alpha_t^2 \sigma^2
    % & + 2\alpha_t p\mathbb{E}[(x^{t,0,0}- x^*)^\top(\nabla \Fm(x^{t,0,0}) )]- 2\alpha_t\sum_{i=0}^{s-1}\sum_{j=0}^{p-1} \mathbb{E}[(x^{t,i,j}- x^*)^\top U_{\sigma(i+1)} \nabla_{\sigma(i+1)} f_{\pi(j)} (x^{t,i,j}) ]
% \end{align*}
+2\alpha_t \mathbb{E}[B^t], 
\end{align}
where $B^t$ is the following term: 
\begin{align*}
     &\left|p(x^{t,0,0}- x^*)^\top(\nabla \Fm(x^{t,0,0}) )- \sum_{i=0}^{s-1}\sum_{j=0}^{p-1} (x^{t,i,j}- x^*)^\top U_{\sigma(i+1)} \nabla_{\sigma(i+1)} f_{\pi(j)} (x^{t,i,j}) \right|.
\end{align*}
Let us analyze the first term in $B^t$ using the definitions of $U_i$ and $\Fm$, writing % $U_i$ in~\eqref{eq:def_Ui}. 
\begin{align*}
     &p(x^{t,0,0}- x^*)^\top(\nabla \Fm(x^{t,0,0}) ) =(x^{t,0,0}- x^*)^\top\left(p\sum_{i=0}^{s-1} U_{\sigma(i+1)} \nabla_{\sigma(i+1)} \Fm (x^{t,0,0}) \right)\\ 
     &= (x^{t,0,0}- x^*)^\top\left(\sum_{i=0}^{s-1}\sum_{j=0}^{p-1} U_{\sigma(i+1)} \nabla_{\sigma(i+1)} f_{\pi(j)} (x^{t,0,0})\right)\\
     &=\sum_{i=0}^{s-1}\sum_{j=0}^{p-1} (x^{t,0,0}- x^*)^\top U_{\sigma(i+1)} \nabla_{\sigma(i+1)} f_{\pi(j)} (x^{t,0,0}).
\end{align*}
% Note that 
% \begin{align*}
%      &\sum_{i=0}^{s-1}\sum_{j=0}^{p-1} (x^{t,0,0}- x^*)^\top U_{\sigma(i+1)} \nabla_{\sigma(i+1)} f_{\pi(j)} (x^{t,0,0}) 
%      =  (x^{t,0,0}- x^*)^\top\left(\sum_{i=0}^{s-1}\sum_{j=0}^{p-1} U_{\sigma(i+1)} \nabla_{\sigma(i+1)} f_{\pi(j)} (x^{t,0,0}) )\right)\\
%      &=  (x^{t,0,0}- x^*)^\top\left(p\sum_{i=0}^{s-1} U_{\sigma(i+1)} \nabla_{\sigma(i+1)} \Fm (x^{t,0,0}) )\right) = p(x^{t,0,0}- x^*)^\top(\nabla \Fm(x^{t,0,0}) ),
% \end{align*}
% where the last line follows from the definition of $U_i$ in~\eqref{eq:def_Ui}.
As a result, $B^t$ is upper bounded by
\begin{align*}
     &\left|\sum_{i=0}^{s-1}\sum_{j=0}^{p-1} (x^{t,0,0}- x^*)^\top U_{\sigma(i+1)} \nabla_{\sigma(i+1)} f_{\pi(j)} (x^{t,0,0}) - \sum_{i=0}^{s-1}\sum_{j=0}^{p-1} (x^{t,0,0}- x^*)^\top U_{\sigma(i+1)} \nabla_{\sigma(i+1)} f_{\pi(j)} (x^{t,i,j}) \right|\\  
    &+\left|\sum_{i=0}^{s-1}\sum_{j=0}^{p-1} (x^{t,0,0}- x^*)^\top U_{\sigma(i+1)} \nabla_{\sigma(i+1)} f_{\pi(j)} (x^{t,i,j}) - \sum_{i=0}^{s-1}\sum_{j=0}^{p-1} (x^{t,i,j}- x^*)^\top U_{\sigma(i+1)} \nabla_{\sigma(i+1)} f_{\pi(j)} (x^{t,i,j}) \right|\\
    &\leq\sum_{i=0}^{s-1}\sum_{j=0}^{p-1}  \left\|x^{t,0,0}- x^* \right\| \cdot \left\|U_{\sigma(i+1)} \nabla_{\sigma(i+1)} f_{\pi(j)} (x^{t,0,0}) -U_{\sigma(i+1)} \nabla_{\sigma(i+1)} f_{\pi(j)} (x^{t,i,j}) \right\|\\  
    &+\sum_{i=0}^{s-1}\sum_{j=0}^{p-1}\left\|x^{t,0,0} -x^{t,i,j}\right\|\cdot\left\| U_{\sigma(i+1)} \nabla_{\sigma(i+1)} f_{\pi(j)} (x^{t,i,j}) \right\|\\
    &\leq\sum_{i=0}^{s-1}\sum_{j=0}^{p-1}  \diam \cdot \Lphif \left \|x^{t,0,0} -x^{t,i,j}\right\|+\sum_{i=0}^{s-1}\sum_{j=0}^{p-1}\left\|x^{t,0,0} -x^{t,i,j}\right\| \cdot \sigma,
\end{align*}
where the last line follows from the $\Lphif$-smoothness and bounded gradient properties of $f_{\pi(j)}$, and the fact that $\left\|x^{t,0,0}- x^* \right\| \leq \diam$.
Applying expectation on both sides of the previous bound, one obtains
\begin{align*}
    &\mathbb{E}[B^t] \leq  (\diam \Lphif +  \sigma) \mathbb{E}\left[\sum_{i=0}^{s-1}\sum_{j=0}^{p-1}  \left \|x^{t,0,0} -x^{t,i,j}\right\|\right] \\
    % &= (\diam \Lphif +  \sigma) \sum_{i=0}^{s-1}\sum_{j=0}^{p-1} \mathbb{E}\left[  \left \|x^{t,0,0} -x^{t,i,0}\right\| +   \left \|x^{t,i,0} -x^{t,i,j}\right\|\right]\\
    &\leq (\diam \Lphif +  \sigma) \sum_{i=0}^{s-1}\sum_{j=0}^{p-1} \left[  (ip + j) \alpha_t\sigma \right]
    = \frac{1}{2}\alpha_t\sigma (\diam \Lphif +  \sigma) \left(s(s-1) p^2 + p^2 s\right) = \frac{1}{2}\alpha_t\sigma (\diam \Lphif +  \sigma) s^2p^2, 
\end{align*}
where the last line follows from the triangle inequality and Lemma \ref{lem:bounds}(b).

Substituting $\mathbb{E}[B^t]$ into our main derivation~\eqref{eq:bound_convex} give us
\begin{align*}
\mathbb{E}[\|x^{t,s,p} - x^*\|^2]
    &\leq \mathbb{E}[\|x^{t,0,0}- x^*\|^2]\\
    &- 2\alpha_t p\mathbb{E}[(x^{t,0,0}- x^*)^\top(\nabla \Fm(x^{t,0,0}) )] + \alpha^2_t sp \sigma[\sigma+  (\diam \Lphif +  \sigma) sp].
\end{align*}
As $f_j$ are convex for $j \in \{ 1, \dots, q\}$, $\Fm$ is convex, and we have 
\begin{align*}
\nabla  \Fm (x^{t,0,0}) ^\top \left(x^{t,0,0} - x^*\right)  \geq \Fm(x^{t,0,0}) - F^*. 
\end{align*}
Applying this bound and $x^{t+1,0,0} = x^{t,s,p}$ yield 
\begin{align*}
\mathbb{E}[\|x^{t+1,0,0} - x^*\|^2]
    &\leq \mathbb{E}[\|x^{t,0,0}- x^*\|^2] - 2\alpha_t p\mathbb{E}[\Fm(x^{t,0,0}) - F^*] + \alpha^2_t sp \sigma[\sigma+  (\diam \Lphif +  \sigma) sp].
\end{align*}

Using the last inequality recursively for all $t$ and rearranging, we obtain 
\begin{align*}
\sum_{t=0}^{T-1} 2\alpha_t p\mathbb{E}[\Fm(x^{t,0,0}) - F^*] 
    &\leq \|x^{0,0,0}- x^*\|^2 - \mathbb{E}[\|x^{T,0,0} - x^*\|^2] +\sum_{t=0}^{T-1} \alpha^2_t sp \sigma[\sigma+  (\diam \Lphif +  \sigma) sp]. 
\end{align*}
Choosing $\alpha_t = \frac{1}{\sqrt{T}}$ results in
\begin{align*}
\frac{1}{T} \sum_{t=0}^{T-1}  \mathbb{E}[\Fm(x^{t,0,0}) - F^*] 
    &\leq \frac{\|x^{0,0,0}- x^*\|^2 + sp \sigma[\sigma+  (\diam \Lphif +  \sigma) sp]}{2p\sqrt{T}} . 
\end{align*}
As $\Fm$ is convex and 
$\Fm(\bar{x})\leq  \frac{1}{T}\sum_{t=0}^{T-1}\Fm (x^{t,0,0})$, we have the desired bound. 

% \begin{align*}
% \Fm(x^{t,i,j+1}) &\leq \Fm(x^{t,i,j}) + \nabla \Fm(x^{t,i,j})^\top (x^{t,i,j+1} - x^{t,i,j}) + \frac{\Lphif}{2} \|x^{t,i,j+1} - x^{t,i,j}\|^2\\
% &\leq \Fm(x^{t,i,j}) - \alpha_t \nabla \Fm(x^{t,i,j})^\top \left(  U_{\sigma(i+1)} \nabla_{\sigma(i+1)} g_{\pi(j)} (x^{t,i,j}, \xi^{t,i,j}) \right)  \\
% &+ \alpha_t^2 \frac{\Lphif}{2} \left\|   U_{\sigma(i+1)} \nabla_{\sigma(i+1)} g_{\pi(j)} (x^{t,i,j}, \xi^{t,i,j})\right\|^2\\
% &\leq \Fm(x^{t,i,j}) - \alpha_t  \nabla_{\sigma(i+1)} \Fm(x^{t,i,j})^\top \left(  \nabla_{\sigma(i+1)} g_{\pi(j)} (x^{t,i,j}, \xi^{t,i,j}) \right)  \\
% &+ \alpha_t^2 \frac{\Lphif}{2} \left\| \nabla_{\sigma(i+1)} g_{\pi(j)} (x^{t,i,j}, \xi^{t,i,j})\right\|^2,
% \end{align*}

% \begin{align*}
%     \|x^{t,i,j} - x^*\|^2 
%     &= \|(x^{t,i,j+1}- x^*)  - (x^{t,i,j+1} - x^{t,i,j})\|^2 \nonumber\\
%     &= \|x^{t,i,j+1}- x^*\|^2 -2 (x^{t,i,j+1}- x^*)^\top(x^{t,i,j+1} - x^{t,i,j}) + \|x^{t,i,j+1} - x^{t,i,j}\|^2.
% \end{align*}
% \begin{align*}
%     \alpha_t(\Fm(x^{t,i,j}) - F^*) \leq \alpha_t \nabla  \Fm (x^{t,i,j}) ^\top \left(x^{t,i,j} - x^*\right)   , 
% \end{align*}

\end{proof}

\bibliography{ref}
\end{document}